\documentclass[a4paper]{article}
\usepackage[mathscr]{eucal}
\usepackage{amssymb}
\usepackage{latexsym}
\usepackage{amsthm}
\usepackage{amsmath}
\usepackage[dvips]{graphicx}
\usepackage{psfrag}

\newtheorem{theorem}{Theorem}[section]

\newtheorem{corollary}[theorem]{Corollary}
\theoremstyle{definition}

\newtheorem{remark}[theorem]{Remark}

\makeatletter
\@addtoreset{equation}{section}

\makeatother

\title{Bounds on three- and higher-distance sets}
\author{
Oleg R.\ Musin\footnote{Research supported in part by NSF grant DMS-0807640 and NSA grant MSPF-08G-201.} 
 \;and 
Hiroshi Nozaki\footnote{Research supported by JSPS Research Fellowship. 
The second author stays at the University of Texas at Brownsville from August 24th, 2009 to August 23rd, 2010.}}

\begin{document}
\maketitle

\renewcommand{\thefootnote}{\fnsymbol{footnote}}

	\begin{abstract}
	A finite set $X$ in a metric space $M$ is called an $s$-distance set if the set of distances between any two distinct points of $X$ has size $s$.
	The main problem for $s$-distance sets is to determine the maximum cardinality of $s$-distance sets for fixed $s$ and $M$. 
	In this paper, we improve the known upper bound for $s$-distance sets in the $n$-sphere for $s=3,4$. In particular, we determine the maximum cardinalities of three-distance sets for $n=7$ and $21$. 
	We also give the maximum cardinalities of $s$-distance sets in the Hamming space and the Johnson space for several $s$ and dimensions. 
	\end{abstract}
	
	
\section{Introduction}
	A finite subset $X$ of the Euclidean space $\mathbb{R}^n$ or the unit sphere $S^{n-1}$ is called an $s$-distance set (or $s$-code) if there exist $s$ Euclidean distances between two distinct vectors in $X$. 
	The main problem for $s$-distance sets is to determine the maximum cardinality of $s$-distance sets for fixed $s$ and $n$. 
	
	Bannai, Bannai and Stanton  \cite{Bannai-Bannai-Stanton} proved that the size of $s$-distance sets in $\mathbb{R}^n$ is bounded above by $\binom{n+s}{s}$. 
	When $s\geq 2$, we know only one example attaining this upper bound, namely, for $(n,s)=(8,2)$ \cite{Lisonek}. 
	The maximum cardinality of $s$-distance sets in $\mathbb{R}^n$ are determined for the following $n$ and $s$ \cite{Croft, Kelly, Lisonek}. 
	
	\[
	\begin{array}{c|ccccccc}
	n &  2 & 3 & 4 & 5 & 6 & 7 & 8\\
	\hline
	\text{size} & 5 & 6 & 10 & 16 & 27 & 29 & 45\\
	\end{array}
	\]
	\begin{center}
		\text{Table 1: Maximum cardinalities of two-distance sets in $\mathbb{R}^n$.}
	\end{center}
	\[
	\begin{array}{c|ccccccc}
	s &  2 & 3 & 4 & 5 \\
	\hline
	\text{size} & 5 & 7 & 9 & 12\\
	\end{array}
	\]
	\begin{center}
	\text{Table 2:  Maximum cardinalities of $s$-distance sets in $\mathbb{R}^2$.}
	\end{center}
	Moreover, Shinohara \cite{Shinohara} proved the icosahedron is the unique maximum three-distance set in $\mathbb{R}^3$.
	
	Delsarte, Goethals, and Seidel proved that the largest cardinality of $s$-distance sets in $S^{n-1}$ is bounded above by $\binom{n+s-1}{s}+\binom{n+s-2}{s-1}$. 
	In the circle, the regular $(2s+1)$-gons attain this upper bound. When $n\geq 3$, we have two examples attaining this upper bound, namely, for $(n,s)=(6,2),(22,2)$ \cite{Delsarte-Goethals-Seidel}. 
	We have the following results for the maximum cardinalities of two-distance sets in $S^{n-1}$ \cite{Delsarte-Goethals-Seidel, Musin}.
	\[
\begin{array}{c|cccccccc}
n     &  2  & 3   & 4    & 5    & 6    & 7 \cdots 21               & 22                   & 24 \cdots 39 \\
\hline
\text{size} & 5 & 6 & 10 & 16 & 27 & \frac{n(n+1)}{2} & 275  & \frac{n(n+1)}{2}                   
\end{array}
	\] 
	\begin{center}
	Table 3:  Maximum cardinalities of two-distance sets in $S^{n-1}$.
	\end{center}
	When $s\geq 3$, we have only one result, namely, that of Shinohara \cite{Shinohara} for $(n,s)=(3,3)$. 
	
	Recently, Musin \cite{Musin} determined the maximum cardinalities of two-distance sets in $S^{n-1}$ for $7\leq n \leq 21$ and $24 \leq n \leq 39$ by a certain general method. 
	This method needs three theorems, namely, Delsarte's linear programming bound, Larman-Rogers-Seidel's theorem and a certain useful bound.  
	This bound in \cite{Musin} is the following: for two-distance sets in $S^{n-1}$ with inner products $a_1$ and $a_2$, 
if $a_1+a_2 \geq 0$, then the size of two-distance set is at most $\binom{n+1}{2}$. 
	Larman, Rogers, and Seidel proved that if the size of a two-distance set in $\mathbb{R}^n$ 
with distances $b_1$ and $b_2$ ($b_1>b_2$)
is greater than $2n+3$, 
then the ratio $b_1^2/b_2^2$ is equal to $k/(k-1)$ where $k$ is a positive integer bounded above by some function of $n$ \cite{Larman-Rogers-Seidel}. 
	This method in \cite{Musin} is applicable to $s$-distance sets in a two-point-homogeneous space $M$ with
a certain assumption. 
	
	Nozaki extended the upper bound in \cite{Musin} to spherical $s$-distance sets for any $s$ \cite{Nozaki-Shinohara}. 
This upper bound is applicable to $M$. By this generalized bound, Barg and Musin \cite
{Barg-Musin} gave 
the maximum $s$-distance sets in the Hamming space and the Johnson space for some $s$ and small dimensions. 
	Larman-Rogers-Seidel's theorem is also extended to $s$-distance sets for any $s$ \cite{Nozaki}. 
	This theorem is also applicable to $s$-distance sets in $M$. 
	
	In the present paper, we improve the known upper bound for $s$-distance sets in $S^{n-1}$ by the method in \cite{Musin} with the generalized Larman-Rogers-Seidel's theorem and the Nozaki upper bound. 
	In particular, we determine the maximum cardinalities of three-distance sets in $S^{7}$ and $S^{21}$. We also give the maximum cardinalities of $s$-distance sets in the Hamming space and the Johnson space for some $s \geq 3$ and more dimensions.

\section{Few distance sets in two-point-homogeneous spaces}

\subsection{Basic definitions}
	In this subsection, we introduce the concept of two-point-homogeneous spaces $M$ 
and our restrictive assumption \cite[Chapter 9]{Conway-Slaone}, \cite{Kabatyansky-Levenshtein, Levenshtein}. 
	
	Let $G$ be a finite group or a connected compact group. 
	We call $M$ a two-point-homogeneous $G$-space if $M$ holds the following properties: 
	\begin{enumerate}
	\item $M$ is a set on which $G$ acts. 
	\item $M$ is a metric space with a distance function $\tau$.
	\item $\tau$ is strongly invariant under $G$: for any $x,x',y,y' \in M$, $\tau(x,y)=\tau(x',y')$ if and only if there is an element $g \in G$ such that $g(x)=x'$ and $g(y)=y'$.   
	\end{enumerate}
	Let $H$ be the subgroup of $G$ that fixes a particular element $x_0 \in M$. 
	Then $M$ can be identified with the space $G/H$ of left cosets $gH$.
	Throughout the present paper, we assume the following: 
	\begin{enumerate}
	\item If $G$ is infinite, then $M$ is a connected Riemannian manifold and 
	$\tau$ is a constant times the natural distance on the manifold. 
	\item If $G$ is finite, and $d_0= \min \tau(x,y)$ for $x,y \in M$, $x\ne y$, 
	then $M$ has the structure of a graph in which $x$ is adjacent to $y$ if and only if $\tau(x,y)=d_0$, 
	and furthermore $\tau$ is a constant times the natural distance in the graph. 
	\end{enumerate}
	Under our assumptions, if $G$ is infinite then Wang \cite{Wang} proved that $M$ is 
	a sphere; real, complex or quaternionic projective space; or the Cayley projective plane. 
	The finite two-point-homogeneous spaces have not yet been completely classified. 
	
	Let $\mu$ be the Haar measure, which is invariant under $G$.
	This induces a unique invariant measure on $M$, which will also  be denoted by $\mu$. 
	We assume that $\mu$ is normalized so that $\mu(M)=1$.
	Let $L^2(G)$ denote the vector space of complex-valued functions $u$ on $G$, satisfying 
	\[
	\int_G |u(g)|^2 d\mu(g) < \infty
	\]
	with inner product 
	\[
	(u_1,u_2)=\int_G u_1(g) \overline{u_2(g)} d\mu(g). 
	\]
	Those $u \in L^2(G)$ that are constant on left cosets of $H$ can be regarded as belonging to $L^2(M)$, 
	which is defined similarly and has the inner product 
	\[
	(u_1,u_2)=\int_M u_1(x) \overline{u_2(x)} d\mu(x). 
	\]
	The space $L^2(M)$ decomposes into a countable direct sum of mutually orthogonal subspaces $\{ V_k \}_{k=0,1,\ldots}$ called (generalized) spherical harmonics. 
	Let $\{ \phi_{k,i} \}_{i=1}^{h_k}$ be an orthonormal basis for $V_k$, where $h_k= \dim V_k$. 
	Since $M$ is distance transitive, the function 
	\[
	\Phi_k(x,y):=\frac{1}{h_k}\sum_{i=1}^{h_k} \phi_{k,i}(x) \overline{\phi_{k,i}(y)}
	\]
	depends only on $\tau(x,y)$. 
	This expression is called the addition formula, and $\Phi_k(\tau)$ is called the zonal spherical function associated with $V_k$.  It is immediate from the definition that $\Phi_k$ is positive definite, that is, 
	\[
	\sum_{x \in X} \sum_{y\in X}\Phi_k(\tau(x,y)) \geq 0
	\]
	for any $X \subset M$.
	For all infinite $M$ and
for all currently known finite cases, $\{\Phi_i \}$ form families of classical orthogonal polynomials. We suppose that the degree of $\Phi_k$ is $k$. Note that $\Phi_k(\tau_0)=1$. 
	
	We define 
	\[
	D(X)=\{\tau(x,y) \mid x,y \in X, x\ne y \}
	\]
	for a finite set $X$ in a two-point-homogeneous space $M$. 
	The finite set $X$ is called an $s$-distance set (or $s$-code) if $|D(X)|=s$. 
	Let $A(M,s)$ be the maximum cardinality of $s$-distance sets in $M$.

\subsection{Delsarte's linear programming bound}    
	The following bound is known as Delsarte's linear programming bound, 
and give a good evaluation for some $D(X)$. 
	\begin{theorem}
	Let $X$ be an $s$-distance set with  $D(X)=\{d_1,d_2,\ldots,d_s\}$. Then
	\begin{align*}
	|X| \leq \max \{1+\alpha_1+\cdots + \alpha_s \mid & \sum_{i=1}^s \alpha_i \Phi_k(d_i) \geq -1, k\geq 0; \\ & \alpha_i \geq 0, i=1,2,\ldots,s \}.
	\end{align*}
	\end{theorem}
	The following is corresponding to the dual problem of the above linear programming problem. 
	\begin{theorem} \label{dual}
	Let $X$ be an $s$-distance set with $D(X)=\{d_1,d_2,\ldots,d_s\}$. Choose a natural number $m$.  Then 
	\begin{align*}
	|X| \leq \min \{1+f_1+ \cdots + f_m \mid & \sum_{k=1}^m f_k \Phi_k(d_i) \leq -1, i = 1,2,\ldots s; \\ 
	&  f_i \geq 0, i=1,2,\ldots,s \}.
	\end{align*}
	\end{theorem}

\subsection{Harmonic absolute bound}
	The following upper bound was proved by Delsarte \cite{Delsarte_1,Delsarte_2,Levenshtein}. 
	\begin{theorem}\label{Delsarte}
	Let $X$ be an $s$-distance set in $M$. Then
	\[
	|X|\leq \sum_{i=0}^s h_i.
	\]
	\end{theorem}
	Nozaki improved the above bound \cite{Nozaki-Shinohara}. 
	\begin{theorem}\label{Nozaki}
	Let $X$ be an $s$-distance set in $M$ with $D(X)=\{d_1,d_2,\ldots,d_s\}$.  
	Consider the polynomial $f(t)=\prod_{i=1}^s (d_i-t)/(d_i-\tau_0)$ and suppose that its expansion in the basis $\{\Phi_k\}$ has the form $f(t)=\sum_{i=0}^s f_i \Phi_i(t)$. 
	Then 
	\[
	|X| \leq \sum_{i:f_i>0} h_i. 
	\]
	\end{theorem}
	When the coefficients $f_i$ are all positive, the bound coincides with the bound in Theorem \ref{Delsarte}.  
	
\subsection{LRS type theorem}
	Let 
	\[
	N(M,s):=h_0+h_1+\cdots +h_{s-1}. 
	\]
	For $d_1,d_2,\ldots ,d_s$, we define the value 
	\[
	K_i:= \prod_{j\ne i} \frac{d_j-\tau_0}{d_j-d_i}
	\] 
	for each $i \in \{1,2,\ldots,s\}$.
	The following theorem is a good constraint to improve the upper bound \cite{Nozaki}. 
	\begin{theorem} \label{LRS}
	Let $X$ be an $s$-distance set in $M$ with $D(X)=\{d_1,d_2,\ldots, d_s\}$. 
	If $|X| \geq 2 N(M,s)$, then 
	$K_i$
	is an integer for each $i \in \{ 1,2,\ldots ,s \}$. Moreover, 
	$
	|K_i| \leq \lfloor 1/2+\sqrt{N(M,s)^2/(2N(M,s)-2) +1/4} \rfloor. 
	$
	\end{theorem}
	The numbers $K_i$ have the following properties. 
	
	\begin{theorem} \label{sum_Ki}
	For any $j \in \{0,1, \ldots s-1\}$, we have $\sum_{i=1}^s d_i^j K_i=\tau_0^j$.
	\end{theorem}
	
	\begin{proof}
	For each $j\in \{1,2,\ldots,s\}$, we define the polynomial 
\[L_j(x):=\sum^s_{i=1}d_i^j\prod_{k\neq i}\frac{x-d_k}{d_i-d_k}\] 
	of degree at most $s-1$. 
	Then the property $L_j(d_i)=d_i^j$ holds for any $i\in \{1,2,\ldots ,s\}$. The polynomial of degree at most $s-1$, that is interpolating distinct $s$ points, is unique. Therefore we can determine $L_j(x)=x^j$.
	\end{proof}
	\begin{corollary} \label{coro}
		\begin{enumerate}
			\item When $s=2$, we have  
			\[
			d_1=\frac{\tau_0-d_2 K_2}{K_1}. 
			\]
			\item When $s=3$, if $d_1>d_2$, then 
			\begin{align*}
				d_1&= \frac{\tau_0 K_1 - d_3 K_1 K_3 - (d_3-\tau_0)\sqrt{- K_1K_2K_3}}{K_1(K_1+K_2)},\\
				d_2&=\frac{\tau_0 K_2- d_3 K_2 K_3 + (d_3-\tau_0)\sqrt{- K_1K_2K_3}}{K_2(K_1+K_2)}.
			\end{align*}
		\end{enumerate}
	\end{corollary}
	\begin{proof}
	We solve the system of equations given by Theorem \ref{sum_Ki}
	\end{proof}
	\begin{remark}
	For $s \geq 4$, there is no simple solution of the system of equations given by Theorem \ref{sum_Ki}.  
	\end{remark}
	\begin{corollary}
	If $d_1>d_2> \cdots >d_s>\tau_0$ $(${\it i.e.} $\tau(\rho)$ is a monotone increasing function$)$ or
 $d_1<d_2<\cdots <d_s<\tau_0$ $(${\it i.e.} $\tau(\rho)$ is a monotone decreasing function$)$, then $|K_1|<|K_2|$. 
	\end{corollary}
	\begin{proof}
	This is immediate because 
	\[
	\left| \frac{K_1}{K_2} \right| = \left| \frac{\tau_0-d_2}{\tau_0-d_1} \cdot \frac{d_3-d_2}{d_3-d_1} \cdot \cdots  \cdot \frac{d_s-d_2}{d_s-d_1}  \right| < 1. 
	\]
	\end{proof}
\subsection{New bounds}
	Let $\mathfrak{D}(M,s)$ be the set of all possible $s$ distances $D(X)=\{d_1,d_2, \ldots ,  d_s \}$ satisfying that $K_i$ are integers.  
	For each $D \in \mathfrak{D}(M,s)$, we have the two bounds, those are the harmonic absolute bound $H(D)$ in Theorem \ref{Nozaki}, and Delsarte's linear programming bound $L(D)$.  
	Then the following immediately holds. 
	
	\begin{theorem} \label{new}
	Let $B(D):=\min \{H(D),L(D) \}$ for $D \in \mathfrak{D}(M,s)$. Then   
	\[
	A(M,s) \leq \max_{D \in \mathfrak{D}(M,s)} \{ B(D), 2N(M,s)-1 \}.
	\]
	\end{theorem}

\section{Bounds on sets with few distances} 
\subsection{Hamming space}
	In this section, we deal with the Hamming space $\mathbb{F}_2^n$ with the Hamming distance $\tau(x,y):=|\{i \mid x_i \ne y_i\}|$ where $x=(x_1,x_2,\ldots, x_n)$ and $y=(y_1,y_2,\ldots,y_n)$. Then $\Phi_k$ is the Krawtchouk polynomial of degree $k$:
	\[
	\Phi_k(x):=\binom{n}{k}^{-1}\sum_{j=0}^k (-1)^j \binom{x}{j} \binom{n-x}{k-j}.
	\]
	We have $h_i=q^{-n}\binom{n}{i}(q-1)^i$. 
	
	When $2s \leq n$, we can construct an $s$-distance set in $\mathbb{F}_2^n$ with $\sum_{i=0}^{\lfloor s/2 \rfloor} \binom{n}{s-2i}$ points. 
	Namely, the example consists of all vectors having $k$ ones for all $k \equiv s \mod 2$. We obtain a lower bound 
	\begin{equation}\label{low_hamming}
	A(\mathbb{F}_2^n,s) \geq \sum_{i=0}^{\lfloor \frac{s}{2} \rfloor} \binom{n}{s-2i} 
	\end{equation}
	for $2s \leq n$. 
	
	Maximum two-distance sets are studied in \cite{Barg-Musin}.
	\begin{theorem}
	If $6 \leq n \leq 74$ with the exception of the values $n=47,53,59,65,70,71$, 
	or if $n=78$, then $A(\mathbb{F}_2^n,2) \leq (n^2-n+2)/2$. 
	\end{theorem}
	
	We determine the maximum cardinalities of three- or four-distance sets in $\mathbb{F}_2^n$ for some $n$. 
	\begin{theorem} \label{hamming}
	\begin{enumerate}
		\item If $8 \leq n \leq 22$, $24 \leq n \leq 33$, or $n=36,37,44$, then $A(\mathbb{F}_2^n,3) = n+ \binom{n}{3}$. 
		\item If $10 \leq n \leq 47$, then $A(\mathbb{F}_2^n,4) = 1+\binom{n}{2}+ \binom{n}{4}$. 
	\end{enumerate}
	\end{theorem}
	\begin{proof}
	In \cite{Barg-Musin} it is proved that (1) for $8 \leq n \leq 22$ and $n=24$, and (2) for $10 \leq n \leq 24$. 
	Since $\mathbb{F}_2^n$ is finite, we can obtain the finite set $\mathfrak{D}(\mathbb{F}_2^n,s)$. We apply Theorem \ref{new} for $\mathfrak{D}(M,s)$. Then this theorem follows from (\ref{low_hamming}). 
	\end{proof}
	
	\begin{remark}
	We also have $A(\mathbb{F}_{2}^{23},3)=2048$, which is obtained from
	the even subcode of the Golay code $\mathcal{G}_{23}$ 
({\it i.e.} the dual code $\mathcal{G}_{23}^{\perp}$ \cite{Barg-Musin, Levenshtein}).
	Our method can be applied for other relatively small $s$.  
	For $s \geq 3$, the authors know no example whose cardinality is greater than the value in the lower bound 
(\ref{low_hamming}) except for $\mathcal{G}_{23}^{\perp}$. 
	\end{remark}

\subsection{Johnson space}
	The binary Johnson space $\mathbb{F}_2^{n,w}$ consists of $n$-dimensional binary vectors with $w$ ones, where $2w \leq n$. The distance is $\tau(x,y)=|\{i \mid x_i \ne y_i \}|/2$. Then $\Phi_k$ is the Hahn polynomial of degree $k$:
	\[
	\Phi_k(x):=\sum_{j=0}^k (-1)^j
	\frac{\binom{k}{j} \binom{n+1-k}{j}}{\binom{w}{j} \binom{n-w}{j}}\binom{x}{j}.
	\]
	We have $h_i=\binom{n}{i}-\binom{n}{i-1}$. 
	 
	When $s \leq n-w$, we can construct $s$-distance sets in $\mathbb{F}_2^{n,w}$ with $\binom{n-w+s}{s}$ points. 
	The example consists of the all vectors with $w-s$ ones in the first coordinates and the remaining $s$ ones anywhere outside them. 
	Therefore we have a lower bound 
	\begin{equation}\label{low_johnson}
	A(\mathbb{F}_2^{n,w},s) \geq \binom{n-w+s}{s}
	\end{equation}
	for $s \leq n-w$.
	 
	 The case $s=2$ was already considered in \cite{Barg-Musin}. 
\begin{theorem}
If $n$ and $w$ satisfy any of the following conditions:
\begin{align*}
6 \leq n \leq 8 \qquad \qquad
&\text{ and } 
w = 3, \\
9 \leq n \leq 11 \qquad \qquad
&\text{ and } 
3 \leq w \leq 4, \\
12 \leq n \leq 14 \text{ or } 25 \leq n \leq 34 
&\text{ and } 
3 \leq w \leq 5, \\
15 \leq n \leq 24 \text{ or } 35 \leq n \leq 46 
&\text{ and } 
3 \leq w \leq 6, 
\end{align*}
then  $A(\mathbb{F}_2^{n,w},2)= (n-w+1)(n-w+2)/2$. 
\end{theorem}
We also have $A(\mathbb{F}_2^{23,7},2)=253$, 
which is obtained from the $253$ vectors of weight $7$ in the binary 
Golay code of length $23$ \cite{Barg-Musin}, \cite[p.\ 69]{book_code}. The code attains the upper bound in Theorem \ref{Delsarte}. 
Let $X$ be the set of the $253$ vectors. 
We can compute an upper bound $A(\mathbb{F}_2^{24,8},2) \leq 253$ by the method in Barg--Musin's paper \cite{Barg-Musin}. 
Though they did not mention the tightness about this bound, an attaining example is easily constructed by  
\[
Y:=\{(1,u) \mid u \in X \}. 
\]
Clearly $Y$ is a two-distance set $\mathbb{F}_2^{24,8}$ with $253$ points, and hence $A(\mathbb{F}_2^{24,8},2) = 253$. 

	 We give the following maximum cardinalities of three- or four-distance sets in $\mathbb{F}_2^{n,w}$ for some $n$ and $w$. 
	 \begin{theorem} \label{s=3_johnson}
	 \begin{enumerate}
	 \item For $11 \leq n \leq 45$ and $4 \leq w \leq n/2$, we have $A(\mathbb{F}_2^{n,w},3) \leq h_0+h_1+h_3=\binom{n}{3}-\binom{n}{2}+n$.
	 \item If $n$ and $w$ satisfy any of the following conditions:
\begin{align*}
11 \leq n \leq 12 
&\text{ and } 
w = 4, \\
13 \leq n \leq 15 
&\text{ and } 
4 \leq w \leq 5, \\
16 \leq n \leq 19 
&\text{ and } 
4 \leq w \leq 6, \\
20 \leq n \leq 24 
&\text{ and } 
4 \leq w \leq 7, \\
25 \leq n \leq 50 
&\text{ and } 
4 \leq w \leq 8, 
\end{align*}
then $A(\mathbb{F}_2^{n,w},3) = \binom{n-w+3}{3}$. 
	\end{enumerate}
	 \end{theorem}
	 \begin{proof}
	 We have the finite set $\mathfrak{D}(\mathbb{F}_2^{n,w},s)$. This theorem is immediate from the bound in Theorem \ref{new} and (\ref{low_johnson}).
	 \end{proof}
	  \begin{theorem}
	 \begin{enumerate}
	 \item For $14 \leq n \leq 58$ and $5 \leq w \leq n/2$, we have $A(\mathbb{F}_2^{n,w},4) \leq h_0+h_1+h_2+h_4=\binom{n}{4}-\binom{n}{3}+\binom{n}{2}$.
	 \item If $n$ and $w$ satisfy any of the following conditions:
\begin{align*}
15 \leq n \leq 16 \qquad \qquad
&\text{ and } 
w = 5, \\
17 \leq n \leq 19 \qquad \qquad
&\text{ and } 
5 \leq w \leq 6, \\
20 \leq n \leq 24 \qquad \qquad
&\text{ and } 
5 \leq w \leq 7, \\
25 \leq n \leq 29 \qquad \qquad
&\text{ and } 
5 \leq w \leq 8, \\
30 \leq n \leq 34 \text{ or } 41 \leq n \leq 47 
&\text{ and } 
5 \leq w \leq 9, \\
35 \leq n \leq 40 \text{ or } 48 \leq n \leq 59 
&\text{ and } 
5 \leq w \leq 10, \\
60 \leq n \leq 70  \qquad \qquad
&\text{ and } 
5 \leq w \leq 11,  
\end{align*}
then $A(\mathbb{F}_2^{n,w},4)= \binom{n-w+4}{4}$. 
	\end{enumerate}
	 \end{theorem}
	\begin{proof}
	 This proof is the same as that of Theorem \ref{s=3_johnson}
	 \end{proof}
	 \begin{remark}
For relatively small $s$, we can obtain  similar results. 
For $s \geq 3$, the authors know no example whose cardinality is greater than the value in the lower bound 
(\ref{low_johnson}).  We can regard a bound for $s$-distance sets in $\mathbb{F}_2^{n,w}$ as that for $w$-uniform $s$-intersecting families \cite{Barg-Musin,1,2,3}.  
	 \end{remark}
	 
\subsection{Spherical space}
	For the unit sphere $S^{n-1}$, we use the usual inner product as $\tau$. Then $\Phi_k$ is the Gegenbauer polynomial of degree $k$. The Gegenbauer polynomials $G_k$ are defined by the following manner: 
	\[
	x G_k(x)=\lambda_{k+1} G_{k+1}(x)+(1-\lambda_{k-1})G_{k-1}(x)
	\]
	where $\lambda_k=k/(n+2k-2)$, $G_0(x) \equiv 1$, and $G_1(x)= nx$. We have $\Phi_k(x)=G_k(x)/h_k$ where $h_k=\binom{n+k-1}{k} - \binom{n+k-3}{k-2}$.
	 
	We can construct an $s$-distance set in $S^{n-1}$ with $\binom{n+1}{s}$ points for $2s \leq n+1$. 
	The example consists of all vectors those are of length $n+1$, and have exactly $s$ entries of $1$ and $n+1-s$ entries of $0$. 
	Since the finite set is on the hyper plane which is perpendicular to the vector of all ones, we can regard it as a subset of $S^{n-1}$. Thus we have a lower bound 
	\begin{equation} \label{eq:low_sphere}
	A(S^{n-1},s) \geq \binom{n+1}{s}
	\end{equation}
	for $2s \leq n+1$. 
	
	The following are new bounds on three- or four-distance sets in $S^{n-1}$ for some $n$. 
	\begin{theorem}\label{s=3_sphere}
	\begin{enumerate}
		\item $A(S^{7},3)=120$ and $A(S^{21},3)=2025$. 
		\item $A(S^{3},3) \leq 27$, $A(S^{4},3) \leq 39$ and $A(S^{6},3) \leq 91$. 
		\item For $n=6$ or $9 \leq n \leq 19$, we have $A(S^{n-1},3) \leq h_1 + h_3= n(n+1)(n+2)/6$. 
		\item For $20 \leq n \leq 30$, we have $A(S^{n-1},3) \leq h_0 + h_1 + h_3= (n+3)(n^2+2)/6$. 
		\item For $31 \leq n \leq 50$, we have $A(S^{n-1},3) \leq h_2+ h_3= (n^2-1)(n+6)/6$. 
	\end{enumerate}
	\end{theorem}
	\begin{proof}
	Let $X \subset S^{n-1}$ be a three-distance set with $D(X)=\{d_1,d_2,d_3\}$ where $d_1<d_2<d_3<\tau_0=1$. 
	By Corollary \ref{coro}, we write 
	\begin{align*}
				d_1&= \frac{ K_1 - d_3 K_1 K_3 - (d_3-1)\sqrt{- K_1K_2K_3}}{K_1(K_1+K_2)},\\
				d_2&=\frac{ K_2- d_3 K_2 K_3 + (d_3-1)\sqrt{- K_1K_2K_3}}{K_2(K_1+K_2)}.
	\end{align*}
	The maximum inner product $d_3$ should be greater than zero. 
	Otherwise the cardinality is smaller than $2n+1$ by Rankin's third bound \cite{Rankin}, 
\cite[page 16]{Ericson-Zinoviev}.  
	Dividing the range $0<d_3<1$ into sufficiently many parts, we obtain finitely many choices of $d_3$. 
	For finitely many choices of three inner products from $K_i$ and $d_3$, we apply Theorem \ref{new}. 
Then the upper bound of $A(S^{n-1},3)$ is obtained numerically. 
	
	For $n=8$ and $n=22$, we have examples attaining the upper bounds. 
	For $n=8$, the examples can be constructed from subsets of the $E_8$ root system. 
	Let $X$ be the $E_8$ root system normalized to have the norm $1$.  
	We have $D(X)=\{ 0, -1, \pm 1/2\}$ and $|X|=240$.  
	There exists $Y \subset X$ such that $Y \cup (-Y)=X$ and $|Y|=|X|/2$. 
	Then, $D(Y)=\{ 0, \pm 1/2\}$, and hence $Y$ is a three-distance set with $120$ points in $S^7$. 	 
	For $n=22$, the example is a subset of the minimum vectors in the Leech lattice. 
	Let $X \subset S^{23}$ be the minimum vectors normalized to have the norm $1$.
	For fixed $x,y \in X$ such that $\tau(x,y)=-1/4$, we obtain
	\[
	Y=\{z \in X \mid \tau(z,x)=1/2, \tau(z,y)=0 \}.
	\]
	Then, $Y \subset S^{21}$ has $2025$ points and $D(Y)=\{7/22,-1/44,-4/11 \}$. 
	\end{proof}
	\begin{remark}
	We have a lot of maximum three-distance sets in $S^7$ up to orthogonal transformations because
	there exist many choices of subsets $Y$ in the above proof. 
    Only one maximum three-distance set in $S^{21}$ is known, and hence it might be unique. 
	\end{remark}
	\begin{remark}
	For the case $s=2$, giving polynomials in Theorem \ref{dual} concretely, we obtained a similar result (see details in \cite{Musin}). 
	We can use this approach also for $s=3$. 
	\end{remark}
	\begin{theorem}  \label{s=4_sphere}
	\begin{enumerate}
		\item $A(S^{4},4) \leq 99$, $A(S^{5},4) \leq 153$ and $A(S^{6},4)\leq 223$. 
		\item For $8 \leq n \leq 15$ or $n=18$, we have $A(S^{n-1},4) \leq h_0+h_2 + h_4= n(n+1)(n+2)(n+3)/24$. 
		\item For $16 \leq n \leq 17$, we have $A(S^{n-1},4) \leq h_0 + h_3 + h_4= (n+3)(n^3+7n^2-10n+8)/24$. 
		\item For $19 \leq n \leq 21$, we have $A(S^{n-1},4) \leq h_2+ h_3= d(n+5)(n^2+n+6)/24$. 
	\end{enumerate}
	\end{theorem}
	\begin{proof}
	The proof of this theorem is the same as that of Theorem \ref{s=3_sphere} except for the way to obtain $d_i$. 
	For given $K_i$ and $d_4$, we find the solutions of the system of equations given by Theorem \ref{sum_Ki} numerically.  
	\end{proof}
	It is possible to calculate for $s\geq 5$ or large $n$, but it takes much time and needs more memory. 
	The following table shows an example whose size is greater than the value in 
the lower bound (\ref{eq:low_sphere}) for $s\geq 3$, and except for $(n,s)=(8,3),(22,3)$.
	\[
\begin{array}{ccccccc} \hline
n  & s & |X|  & \text{inner products}              & \text{absolute bound} & \text{new bound} & \text{bound (\ref{eq:low_sphere})}  \\
\hline
23 & 3 & 2300  &   0, \pm \frac{1}{3}                     & 2576           & 2301    & 2024 \\
8  & 4 & 240   &   -1,0, \pm \frac{1}{2}             &  450           & 330     & 126   \\
24 & 5 & 98280 & 0, \pm \frac{1}{4}, \pm \frac{1}{2} &115830          & \text{?}       & 53130 \\
24 & 6 & 196560& -1,0, \pm \frac{1}{4}, \pm \frac{1}{2} & 573300     & \text{?}     & 177100 \\
\hline
\end{array}
	\] 
	The examples in the above table are obtained from 
tight spherical designs, or their subsets \cite{Delsarte-Goethals-Seidel, Levenshtein}. 
The methods in Theorems \ref{s=3_sphere} and \ref{s=4_sphere} are applicable to other projective spaces.   
	
\begin{remark}
Our method is applicable to a $Q$-polynomial association scheme defined in \cite{Delsarte_1} (also see \cite{Bannai-Ito}). 
A $Q$-polynomial association scheme is not always a two-point-homogeneous space.  
There are two concepts which include the projective spaces and $Q$-polynomial association schemes, 
namely, $Q$-polynomial spaces \cite{Godsil} and Delsarte spaces \cite{Neumaier}.   
The method in the present paper is applicable to both of the two concepts. 
\end{remark}
	
	\noindent
	\textbf{Acknowledgments.} The authors thank Alexander Barg, Grigori Kabatianski, Eiichi Bannai, Masashi Shinohara and Sho Suda for useful discussions and comments.

{\it Oleg R.\ Musin}\\
	Department of Mathematics, \\
	University of Texas at Brownsville, \\
	80 Fort Brown, \\
	Brownsville, \\
	TX 78520, \\
	USA. \\
	omusin@gmail.com \\
\quad \\
{\it Hiroshi Nozaki}\\
	Graduate School of Information Sciences, \\
	Tohoku University \\
	Aramaki-Aza-Aoba 6-3-09, \\
	Aoba-ku, \\
	Sendai 980-8579, \\
	Japan.\\ 
	nozaki@ims.is.tohoku.ac.jp\\

\end{document}